\documentclass[
final
, nomarks
]{dmtcs-episciences}

\usepackage{amsmath} 

\usepackage[utf8]{inputenc}
\usepackage{subfigure}

\newcommand{\PP}{\mathbb{P}}
\newcommand{\E}{\mathbb{E}}
\newcommand{\B}{\mathbf B}
\newcommand{\Z}{\mathbb{Z}}
\newcommand{\C}{\mathbb{C}}
\newcommand{\R}{\mathbb{R}}

\usepackage{amssymb}
\usepackage{amscd}

\usepackage{amsthm}
\usepackage{amsfonts}

\theoremstyle{plain}
\newtheorem{theorem}{Theorem}[section]
\newtheorem{lemma}[theorem]{Lemma}

\newtheorem{corollary}[theorem]{Corollary}

\theoremstyle{definition}

\theoremstyle{remark}

\numberwithin{equation}{section}

\DeclareMathAlphabet{\mathbbold}{U}{bbold}{m}{n}

\usepackage[round]{natbib}

\author[Charles Burnette]{Charles Burnette}
\title[Involution factorizations of Ewens random permutations]{Involution factorizations of Ewens random permutations}
\affiliation{Xavier University of Louisiana, New Orleans, USA}
\begin{document}
\publicationdata{vol. 27:2}{2025}{11}{10.46298/dmtcs.11602}{2023-07-19; 2023-07-19; 2024-12-24}{2025-03-28}
\maketitle
\begin{abstract}
\vspace*{2ex}
 An involution is a bijection that is its own inverse.  
Given a permutation $\sigma$ of $[n],$ let $\mathsf{invol}(\sigma)$ denote the number 
of ways $\sigma$ can be expressed as a composition of two involutions of $[n].$
We prove that the statistic $\mathsf{invol}$ is asymptotically lognormal when the
symmetric groups $\mathfrak{S}_n$ are each equipped with Ewens Sampling Formula
probability measures of some fixed positive parameter $\theta.$
This paper strengthens and generalizes previously determined results about
the limiting distribution of $\log(\mathsf{invol})$ for uniform random permutations,
i.e. the specific case of $\theta = 1$. We also investigate the first two moments of
$\mathsf{invol}$ itself, detailing the phase transition in asymptotic behavior at $\theta = 1,$ and provide a functional refinement and a demonstrably optimal convergence rate for the log-Gaussian limit law.
\end{abstract}

\section{Introduction}
\label{sec:in}
An involution is a bijection that is its own inverse. The involutions in the symmetric group $\mathfrak{S}_n$
are precisely those permutations whose cycles are all of length 1 or 2.
This paper is concerned with decomposing random permutations into two involutive factors
and the probabilistic analysis thereof. Recent interest in involution factorizations in $\mathfrak{S}_n$
can be traced to \cite{RV} who proposed that the dynamical behavior of reversible
rational maps over a finite field can be modeled by compositions of two random involutions with a
prescribed number of fixed points. Unrestricted products of two involutions, however, are of
singular interest as well. For instance, \cite{PT} outline hypothetical links between
enumerating involution factorizations, rim-hook tableaux, and the irreducible characters of
$\mathfrak{S}_n$ that are still in need of further clarification. Other applications have surfaced in
comparative genomics (e.g. \cite{BEF}, \cite{CM}) and the analysis of algorithms (e.g. \cite{KSY} and \cite{YEMR}).

If $\sigma$ is a permutation of $[n],$
let $\mathsf{invol}(\sigma)$ denote the number of function compositional factorizations of $\sigma$ into
the product of two involutions of $[n],$ that is, the number of ordered pairs $(\tau_1, \tau_2)$ of involutions in $\mathfrak{S}_n$
such that $\sigma = \tau_2 \circ \tau_1.$ It is known that

\begin{equation}
\label{Nproduct}
\mathsf{invol}(\sigma) = \B(\sigma)\prod_{k = 1}^{n}\sum\limits_{j=0}^{\lfloor c_{k}/2\rfloor}
\frac{(c_{k})_{2j}}{(2k)^{j}j!},
\end{equation}
where $(x)_r$ is the falling factorial $x(x - 1)\cdots(x - r + 1),$ $c_{k}:=c_{k}(\sigma)$
denotes the number of cycles of length $k$ (or, more succinctly, $k$-cycles)
that $\sigma$ possesses, and $\B(\sigma) = \prod k^{c_k}$ is the product of its cycle lengths.
To the author's knowledge, the earliest proofs of (\ref{Nproduct}) are independently
attributed to \cite{LugoCompose} and \cite{PT}.
Both of their derivations rely on identifying involutions with partial matchings and, consequently,
the composition of two involutions with dichromatic edge-colored graphs
in which every vertex is incident to exactly two edges, one edge per color.
A derivation of (\ref{Nproduct}) that merely exploits the canonical group action of $\mathfrak{S}_n$
on $[n]$ appears in \cite{BS}.

Even without (\ref{Nproduct}), one notices that $\mathsf{invol}(\sigma)$ is a class function by virtue
of the conjugacy invariance of the set of involutions. It is therefore natural to wonder
what the distribution of $\mathsf{invol}$ is for permutations drawn from
conjugation-invariant probability measures on $\mathfrak{S}_n$
where cycles are multiplicatively weighted.
The archetypal example of such a measure is the Ewens Sampling Formula
with parameter $\theta > 0,$ hereafter abbreviated $\textup{ESF}(\theta).$
Read the survey article by \cite{Crane} for a lucid and concise synopsis of the literature on $\textup{ESF}(\theta).$

For each positive integer $n,$ the Ewens probability measure $\mathbb{P}_{\theta, n}$ is a deformation of the uniform measure $\nu_n$ on $\mathfrak{S}_n$ obtained by incorporating the following sampling bias:
\begin{equation}
\mathbb{P}_{\theta, n}(\sigma = \sigma_0) = \frac{n!\theta^K}{\theta^{(n)}}\nu_n(\sigma = \sigma_0) = \frac{\theta^K}{\theta^{(n)}},
\end{equation}
where $\theta^{(n)}$ is the rising factorial $\theta(\theta + 1)\cdots(\theta + n - 1)$ and $K = \sum c_k$ is the total number of cycles that $\sigma$ has. (Note, in particular, that $\mathbb{P}_{1, n} = \nu_n.$) More specifically,
the probability that a permutation of $[n]$ chosen according to $\textup{ESF}(\theta)$ has cycle index $(a_1, \ldots, a_n) \in \Z_{\geq 0}^n$ is
\begin{equation}
\label{Ewensdistribution}
\PP_{\theta,n}(c_k = a_k\ \text{for}\ k = 1, \ldots, n) = \mathbbold{1}\left\{\sum_{\ell = 1}^{n} \ell a_{\ell} = n\right\}\frac{n!}{\theta^{(n)}}\prod_{k = 1}^{n} \left(\frac{\theta}{k}\right)^{a_k}\frac{1}{a_{k}!},
\end{equation}
where $\mathbbold{1}\{\cdot\}$ denotes the indicator function. We shall also refer to random
permutations chosen according ot $\textup{ESF}(\theta)$ as $\theta$-weighted permutations, as is also done in \cite{Lugo}.

A pivotal feature of $\text{ESF}(\theta)$ is that the joint distribution of the cycle counts
$c_k$ converges to that of independent Poisson random variables of respective intensity $\theta/k.$
Now consider a sequence $\alpha = (\alpha_1, \alpha_2, \ldots)$ of mutually independent
random variables with $\alpha_k \sim \text{Poisson}(\theta/k),$ for $k = 1, 2, \ldots,$ and define
\begin{equation}
\mathsf{invol}_n(\alpha) = \B_n(\alpha)\prod_{k = 1}^{n}\sum\limits_{j=0}^{\lfloor \alpha_{k}/2\rfloor}
\frac{(\alpha_{k})_{2j}}{(2k)^{j}j!}, \hspace{2.5em} \text{where} \hspace{2.5em} \B_n(\alpha) = \prod_{k = 1}^{n} k^{\alpha_k}.
\end{equation}
Note that, as $n \to \infty,$
\begin{align}
\label{meandifference}
\E|\!\log\mathsf{invol}_n(\alpha) - \log\B_n(\alpha)| &= \sum_{k = 1}^{n} \E\log\!\left(\sum\limits_{j=0}^{\lfloor \alpha_{k}/2\rfloor}
\frac{(\alpha_{k})_{2j}}{(2k)^{j}j!}\right)
\\ &\leq \sum_{k = 1}^{n} \frac{1}{2k}\E \alpha_k^2 = O(1),
\end{align}
whereas
\begin{equation}
\label{var}
\text{Var}(\log\B_n(\alpha)) = \sum_{k = 1}^{n} \frac{\theta\log^2\!k}{k} = \frac{\theta}{3}\log^3\!n(1 + o(1)).
\end{equation}
Markov's inequality thus dictates that
\begin{equation}
\frac{|\log\mathsf{invol}_n(\alpha) - \log\B_n(\alpha)|}{\sqrt{\text{Var}(\log\B_n(\alpha))}} \overset{\mathbb{P}}{\longrightarrow} 0.
\end{equation}
One would then suspect that something similar holds for $\log\mathsf{invol}$ and $\log\B$ in $\text{ESF}(\theta),$ assuming both are plausibly well-approximated by their Poissonized counterparts.

In this paper, we will establish both central limit and functional limit theorems for $\log\mathsf{invol}$
over Ewens random permutations. This will be accomplished by generalizing
the key results of \cite{BS}. It was particularly shown in \cite{BS} that $\mathsf{invol}$
is asymptotically lognormal when each $\mathfrak{S}_n$ is endowed with the uniform probability
measure. The proof of this hinged on the observation that, for asymptotically almost all permutations
$\sigma \in \mathfrak{S}_n,$ the statistic $\mathsf{invol}(\sigma)$ is readily controlled by
the simpler statistic $\B(\sigma)$ together with the fact that $\B$ is itself an asymptotically lognormal
random variable over uniform random permutations. As we shall soon witness,
these facts remain true in the setting of $\textup{ESF}(\theta).$ The polynomial factors of (\ref{Nproduct}) ony influence the remainder terms in the weak limit theorems stated in Section \ref{sec:refine}.

On the other hand, the distribution of $\mathsf{invol}$ itself, when left untouched by
any rescaling or standardization, is heavily right-skewed.
Practically all of its range is concentrated far below its mean.
To demonstrate this, we will derive the leading-term asymptotics of
the first and second moments of $\mathsf{invol}$ for general $\theta$-weighted permutations.
The latter will highlight certain discrepancies in the variability of
$\mathsf{invol}$ for $0 < \theta \leq 1$ versus $ \theta > 1.$

We remark that the limiting distribution of $\log \mathsf{invol}$ falls within the scope of the general theory of additive functionals on the cycle multiplicities of random permutations (see e.g. \cite{ManstaviciusOriginal}, \cite{BM}). The original contribution of the present paper, however, is the singularity analysis of the generating functions that are used to concretely describe the distribution of $\mathsf{invol}.$ Moreover, all of the fine moment asymptotics that can be found throughout Sections \ref{sec:expectedtypical} through \ref{sec:refine} have not appeared in the literature before.

Some classical results and techniques from asymptotic analysis are cited in a piecemeal
fashion throughout this paper. As is common for random permutation and integer partition
statistics, we will encounter generating functions that demand a careful examination
by way of the saddle-point method and Mellin transforms. Technical preliminaries are
briefly discussed in the appendix, Section \ref{sec:appendix}, at the end.

\section{The Expected and Typical Values of $\mathsf{invol}$}
\label{sec:expectedtypical}
Let $\PP_n$ denote the probability measure on $\mathfrak{S}_n$
corresponding to $\textup{ESF}(\theta).$ For the sake of notational brevity,
we will henceforth omit the $\theta$ in $\PP_{\theta, n}$ and associated operators.

\begin{theorem}
\label{average}
The average number of involution factorizations for a permutation
chosen according to the Ewens measure $\PP_n$ is
\[\E_n\mathsf{invol} = \frac{\theta^{\frac{1 - \theta^2}{4}}\Gamma(\theta)}{2^{\frac{\theta^2}{2} + 1}e^{\frac{\theta}{2}}\sqrt{\pi}}e^{2\sqrt{\theta n}}n^{\frac{\theta^2}{4} - \theta + \frac{1}{4}}\!\left(1 + O(n^{-\frac{1}{2}})\right).\]
\end{theorem}

\begin{proof}
For each $\lambda \in \Z_{\geq 0}^n,$ set $\mathcal{A}_{\lambda} = \{\sigma \in \mathfrak{S}_n : \sigma\ \text{has cycle type}\ 1^{\lambda_1} \cdots n^{\lambda_n}\}.$ Consider the power series
\begin{equation}
F(z) = 1 + \sum_{n = 1}^{\infty} (\E_{n}\mathsf{invol})\frac{\theta^{(n)}}{n!}z^n.
\end{equation}
By the law of total expectation,
\begin{equation}
F(z) = 1 + \sum_{n = 1}^{\infty} \left(\sum_{\lambda \vdash n} \PP_{n}(\sigma \in \mathcal{A}_{\lambda})\E_{n}(\mathsf{invol}(\sigma)\, |\, \sigma \in \mathcal{A}_{\lambda})\right)\frac{\theta^{(n)}}{n!}z^n,
\end{equation}
where $\lambda \vdash n$ means that $\lambda \in \Z_{\geq 0}^n$ and $\sum k\lambda_k = n.$ As pointed out in Theorem 2.5 of \cite{BS}, 
\begin{equation}
\mathsf{invol}(\sigma) = \prod_{k = 1}^{n} k^{c_k(\sigma)}\frac{He_{c_k(\sigma)}(i\sqrt{k})}{(i\sqrt{k})^{c_k(\sigma)}},
\end{equation}
where $He_m$ is the probabilists' Hermite polynomial
$He_{m}(x) = m!\sum_{r = 0}^{\lfloor m/2 \rfloor} 
\frac{(\textendash1)^r}{r!(m - 2r)!}\frac{x^{m - 2r}}{2^r}.$ Thus
\begin{align}
F(z) &= 1 + \sum_{n = 1}^{\infty} \left(\sum_{\lambda \vdash n} \left(\frac{n!}{\theta^{(n)}} \prod_{k = 1}^{n} \left(\frac{\theta}{k}\right)^{\lambda_k}\frac{1}{\lambda_k!}\right)\!\left(\prod_{k = 1}^{n} k^{\lambda_k}\frac{He_{\lambda_k}(i\sqrt{k})}{(i\sqrt{k})^{\lambda_k}}\right)\right)\frac{\theta^{(n)}}{n!}z^n
\\ &= \prod_{k = 1}^{\infty} \left(\sum_{\lambda_k = 0}^{\infty} \frac{He_{\lambda_k}(i\sqrt{k})}{\lambda_k!}\left(\frac{\theta z^k}{i\sqrt{k}}\right)^{\lambda_k}\right)
\end{align}
Invoking the exponential generating function $\exp(xt - t^2/2) = \sum\limits_{n = 0}^{\infty} He_n(x)\frac{t^n}{n!}$ (an assortment of derivations are available in \cite{KKKD}, \cite{Kim}, \cite{KM}), we see that
\begin{equation}
\label{bgf}
F(z) = \prod_{k = 1}^{\infty} \exp\!\left(\theta z^k + \frac{\theta^2z^{2k}}{2k}\right) = \frac{\exp(\theta z/(1 - z))}{(1 - z^2)^{\frac{\theta^2}{2}}}.
\end{equation}
Hence
\begin{equation}
\E_{n}\mathsf{invol} = \frac{n!}{\theta^{(n)}}[z^n]\frac{\exp(\theta z/(1 - z))}{(1 - z^2)^{\frac{\theta^2}{2}}}.
\end{equation}
Applying Corollary \ref{Wrightasymptotic} with $\beta = -\theta^2\!/2,$
$\phi(z) = e^{-\theta}(1 + z)^{-\frac{\theta^2}{2}},$ and $\alpha = \theta,$
together with the Stirling's approximations $n! = \sqrt{2\pi n}e^{-n}n^n\left(1 + O(n^{-1})\right)$
and $\theta^{(n)} = \frac{\Gamma(\theta + n)}{\Gamma(\theta)} = \frac{\sqrt{2\pi}}{\Gamma(\theta)}e^{-n}n^{n + \theta - \frac{1}{2}}\left(1 + O(n^{-1})\right)$ completes the proof.
\end{proof}

We can further adapt the function constructed in the proof of Theorem \ref{average} to
compute the average number of involution factorizations for permutations with a fixed
number of cycles. Indeed, if we set
\begin{equation}
F(u, z) = 1 + \sum_{n = 1}^{\infty} \sum_{m = 1}^{n} \E_{1, n}\!\left(\mathbbold{1}\{K(\sigma) = m\}\mathsf{invol}(\sigma)\right)u^mz^n,
\end{equation}
then
\begin{align}
F(u, z) &= 1 + \sum_{n = 1}^{\infty} \sum_{m = 1}^{n} \sum_{\substack{\lambda \vdash n \\ \sum \lambda_k = m}} \left(\prod_{k = 1}^{n} \left(\frac{1}{k}\right)^{\lambda_k}\frac{1}{\lambda_k!}\right)\!\left(\prod_{k = 1}^{n} k^{\lambda_k}\frac{He_{\lambda_k}(i\sqrt{k})}{(i\sqrt{k})^{\lambda_k}}\right)\!\left(\prod_{k = 1}^{n} u^{\lambda_k}\right)z^n
\\ &= 1 + \sum_{n = 1}^{\infty} \sum_{\lambda \vdash n}\left(\prod_{k = 1}^{n} \left(\frac{u}{k}\right)^{\lambda_k}\frac{1}{\lambda_k!}\right)\!\left(\prod_{k = 1}^{n} k^{\lambda_k}\frac{He_{\lambda_k}(i\sqrt{k})}{(i\sqrt{k})^{\lambda_k}}\right)z^n
\\ &= \frac{\exp(uz/(1 - z))}{(1 - z^2)^{\frac{u^2}{2}}} = \exp\!\left(\frac{uz}{1 - z} + \frac{1}{2}u^2\log\frac{1}{1-z^2}\right)
\end{align}
is the bivariate ordinary generating function for the expected value of
$\mathbbold{1}\{K(\sigma) = m\}\mathsf{invol}(\sigma)$ over uniform
random permutations, counted according to the size
of the ground set (indicated by the variable $z$) and the total
number of cycles (indicated by the variable $u$). The diagonal of
$F(u, z)$ is given by
\begin{equation}
\lim_{z \to 0} F(x/z, z) = e^{x + \frac{x^2}{2}},
\end{equation}
which is the exponential generating function for the class of involutions.
This makes sense, as the only permutation of $[n]$ with $n$ cycles is the
identity permutation, which is precisely the square of every involution in
$\mathfrak{S}_n.$ By expanding the exponential series in (2.11), we can
work towards retrieving the conditional expectation of $\mathsf{invol}$
for permutations with $m$ cycles by using the ``vertical'' generating function
\begin{equation}
[u^m]F(u, z) = \sum_{k = \lceil m/2 \rceil}^{m} \frac{\binom{k}{m - k}}{k!2^{m - k}}\left(\frac{z}{1 - z}\right)^{2k - m}\log^{m - k}\!\left(\frac{1}{1 - z^2}\right).
\end{equation}
For example, $[u]F(u, z) = z/(1 - z),$ and so $[uz^n]F(u, z) = 1$ for
$n \geq 1.$ This too makes sense as the probability of picking an
$n$-cycle is $1/n,$ and there are exactly $n$ ways to write an
$n$-cycle as a product of two involutions. 

\begin{theorem}
For each positive integer $m,$
\[\E_n(\mathsf{invol}(\sigma)\, |\, K(\sigma) = m) = \frac{n^m}{m!\log^{m - 1}\!n}\!\left(1 + O\!\left(\frac{1}{\log n}\right)\right)\]
as $n \to \infty.$
\end{theorem}

\begin{proof}
Because $K$ is a sufficient statistic for the underlying parameter $\theta,$ which means that the conditional distribution of $\mathsf{invol}$ does not depend on $\theta$ (see section 2.2 of \cite{Crane}), the conditional expectation is constant across the support of $K.$ Therefore
\begin{equation}
\E_n(\mathsf{invol}(\sigma)\, |\, K(\sigma) = m) = \frac{[u^mz^n]F(u, z)}{{n\brack m}/n!},
\end{equation}
where ${n\brack m}$ are the unsigned Stirling numbers of the first kind. Based on the Flajolet-Odlyzko transfer theorem (see \cite{FO})
\begin{equation}
[z^n](1 - z)^{\alpha}\log^k\left(\!\frac{1}{1 - z}\right) = \frac{1}{\Gamma(-\alpha)}n^{-\alpha - 1}\log^k\!n + O\!\left(n^{-\alpha - 2}\log^k\!n\right)\ \text{for}\ \alpha \not\in \Z_{\geq 0},
\end{equation}
we find that
\begin{equation}
[u^mz^n]F(u,z) = \frac{1}{m!}[z^n]\!\left(\frac{z}{1 - z}\right)^m + O(n^{m - 2}) = \frac{n^{m - 1}}{m!(m - 1)!} + O(n^{m - 2})
\end{equation}
Evoke the following crude simplification of Wilf's asymptotic (\cite{Wilf})
\begin{equation}
\frac{{n\brack m}}{n!} = \frac{\log^{m - 1}\!n}{(m - 1)!n} + O\!\left(\frac{\log^{m - 2}\!n}{n}\right)
\end{equation}
to finish the calculation.
\end{proof}

Most of the contribution to the mean comes from a small subset of permutations that admit an
exceptionally large number of involution factorizations. Scrupulous glances at (\ref{Nproduct})
should convince the reader that giving a permutation $\sigma$ numerous occurrences of a distinct long cycle
length or an exorbitant number of short cycles exaggerates the magnitude of $\mathsf{invol}(\sigma).$
But it is rare for a permutation to have such a cycle profile. The \textit{typical} permutation has much fewer
involution factorizations. Direct combinatorial arguments can verify this for uniform random permutations,
which have roughly $e^{2\sqrt{n}}/\sqrt{8\pi en}$ involution factorizations on average whereas the median is $O(e^{\log^2\!n}),$ which we will show in Theorem \ref{skew} and its corollary. (In fact, it is closer to $e^{\frac{1}{2}\log^2\!n},$ as shown in \cite{BS}.) For this, we use the following lemma.

\begin{lemma}
\label{partitionnorm}
For each $\lambda \in \Z_{\geq 0}^n,$ set $\mathcal{A}_{\leq\lambda} = \{\sigma \in \mathfrak{S}_n : c_k(\sigma) \leq \lambda_k\ \text{for every}\ k \in [n]\}.$ Then $\max_{\sigma \in \mathcal{A}_{\leq\lambda}} \mathsf{invol}(\sigma)$ is attained at a permutation with the maximally obtainable amount of shortest cycles.
\end{lemma}

For a cursory explanation as to why, note that if $c_j(\sigma) \geq 1$ and $\lambda = \langle \lambda_1, \lambda_2, \ldots, \lambda_j \rangle$ is a partition of $j$ with $\lambda_1 \neq 1,$ then any permutation obtained by splitting up each $j$-cycle of $\sigma$ into $\lambda_r$ $r$-cycles for $r = 1, 2, \ldots, j$ generally enlarges the number of involution factorizations. (Cf the first two paragraphs in the proof of Theorem 16 in \cite{SS} which discuss how this strategy optimizes the product of the parts of a partition.)

In what follows, we let $\xi_1$ and $\xi_2$ be integers such that $20 \leq \xi_1, \xi_2 \leq n,$ and $\xi_1T_{\xi_2} < n,$ where $T_r$ is the $r^{\text{th}}$ triangular number. Let $\mathcal{P}_{\xi_1, \xi_2}$ be the set of all permutations $\sigma \in \mathfrak{S}_n$ for which $c_k(\sigma) \leq \xi_1$ for all $k \leq \xi_2$ and $c_k(\sigma) \leq 1$ for all $k > \xi_2.$ Observe that the condition $\xi_1T_{\xi_2} < n$ forces each permutation in $\mathcal{P}_{\xi_1, \xi_2}$ to have at least one cycle longer than $\xi_2.$ Lemma III of \cite{ET1} states that
\begin{equation}
\label{ETprob}
\nu_n(\sigma \not\in \mathcal{P}_{\xi_1, \xi_2}) \leq e^{\frac{15}{(\xi_1 + 1)!} + \frac{3}{\xi_2}} - 1.
\end{equation}
We can adapt this result to bound the median number of involution factorizations in uniform random permutations.

\begin{theorem}
\label{skew}
If $H_n$ denotes the $n^{\text{th}}$ harmonic number and
\[s_n^2 = 2\sum_{j = 1}^{n - 1} \frac{H_j}{j + 1} - H_n(H_n - 1),\]
then at least $n!(2 - 1/\xi_1^2 - \exp(15/(\xi_1 + 1)! + 3/\xi_2))$ permutations $\sigma \in \mathfrak{S}_n$ satisfy the inequality
\[\mathsf{invol}(\sigma) < (\xi_2!)^{\xi_1}n^{\lceil H_n + \xi_1s_n\rceil - \xi_2}e^{\frac{1}{2}\xi_1^2H_{\xi_2}}.\]
\end{theorem}

\begin{proof}
Let $\kappa(\sigma)$ denote the number of distinct cycle lengths present in the cycle decomposition of $\sigma,$ and let $\mathcal{E}_{\xi_1}$ be the event that $\kappa(\sigma) \leq H_n + \xi_1s_n.$ It is folklore that $H_n$ and $s_n$ are the mean and standard deviation, respectively, of $K$ over uniform random permutations of $[n].$ This together with (\ref{ETprob}), the union bound, and Chebyshev's inequality tell us that
\begin{align}
\nu_n(\sigma \not\in \mathcal{P}_{\xi_1, \xi_2} \cap \mathcal{E}_{\xi_1}) &\leq e^{\frac{15}{(\xi_1 + 1)!} + \frac{3}{\xi_2}} - 1 + \nu_n(\kappa(\sigma) > H_n + \xi_1s_n)
\\ &\leq e^{\frac{15}{(\xi_1 + 1)!} + \frac{3}{\xi_2}} - 1 + \nu_n(|K(\sigma) - H_n| \geq \xi_1s_n)
\\ &\leq e^{\frac{15}{(\xi_1 + 1)!} + \frac{3}{\xi_2}} - 1 + \frac{1}{\xi_1^2}.
\end{align}

Now let $\sigma^{\text{max}}$ be a permutation for which $\mathsf{invol}$ attains its maximum value over $\mathcal{P}_{\xi_1, \xi_2} \cap \mathcal{E}_{\xi_1}.$ By Lemma \ref{partitionnorm}, if we set $\xi_3 = \lceil H_n + \xi_1s_n\rceil,$ then $c_k(\sigma^{\text{max}}) = \xi_1$ for $1 \leq k \leq \xi_2,$ and $c_k(\sigma^{\text{max}}) = 1$ for up to an additional $\xi_3 - \xi_2$ larger integers $k.$ Hence,
\begin{align}
\mathsf{invol}(\sigma^{\text{max}}) &= \B(\sigma^{\text{max}})\prod_{k = 1}^{n} \sum\limits_{j=0}^{\lfloor c_{k}(\sigma^{\text{max}})/2\rfloor}\frac{(c_{k}(\sigma^{\text{max}}))_{2j}}{(2k)^{j}j!}
\\ &\leq \left(\prod_{k = 1}^{\xi_2} k^{\xi_1}\right)\cdot\left(\prod_{k = \xi_2 + 1}^{\xi_3} n\right)\cdot\left(\prod_{k = 1}^{\xi_2} \sum\limits_{j=0}^{\lfloor \xi_2/2\rfloor}\frac{(\xi_1)_{2j}}{(2k)^{j}j!}\right)
\\ &\leq (\xi_2!)^{\xi_1}\cdot n^{\xi_3 - \xi_2}\cdot\prod_{k = 1}^{\xi_2}\exp\!\left(\frac{\xi_1^2}{2k}\right) = (\xi_2!)^{\xi_1}n^{\xi_3 - \xi_2}e^{\frac{1}{2}\xi_1^2H_{\xi_2}}
\end{align}
\end{proof}

Setting $\xi_1 = \xi_2$ in Theorem \ref{skew} and employing the Stirling bound from \cite{Robbins}
\[n! < \sqrt{2\pi}n^{n+\frac{1}{2}}e^{-n + \frac{1}{12n}}\]
along with the elementary bounds $H_n < \log n + 1$ and $s_n < \sqrt{\log n}$ begets the following.

\begin{corollary}
\label{skewer}
If $\xi = \xi(n) \to \infty,$ but slowly enough so that $\xi = o\!\left(\sqrt{\log n}\right),$ then all but $o(n!)$ permutations $\sigma \in \mathfrak{S}_n$ satisfy the inequality
\[\mathsf{invol}(\sigma) < e^{(1 + O(\xi\log^{-1/2}\!n))\log^2\!n}.\]
\end{corollary}

Goncharov's limit law for $K$ over $\textup{ESF}(\theta)$ (Theorem 3.5, \cite{ABT1}) and the upcoming Lemma \ref{probabilistic} indicate that, assuming the same notations of the previous proof, asymptotically almost all $\theta$-weighted permutations satisfy the analogous inequality
\[\mathsf{invol}(\sigma) < e^{(\theta + O(\xi\log^{-1/2}\!n))\log^2\!n}.\]
However, we can conclude something even stronger. Based on Theorem \ref{mainthm}, if $\sigma$ is a $\theta$-weighted permutation of $[n]$ and $\epsilon > 0,$ then w.h.p.,
\[e^{(\frac{\theta}{2} - \epsilon)\log^2\!n} < \mathsf{invol}(\sigma) < e^{(\frac{\theta}{2} + \epsilon)\log^2\!n}.\]

\section{A Close Look at the Second Moment of $\mathsf{invol}$}
\label{sec:secondmoment}
Ultimately, the asymptotic lognormality of $\mathsf{invol}$ for $\theta$-weighted permutations
prevents, say, the method of moments from succeeding.
(See, for instance, Example 30.2 on pg. 389 of \cite{Billingsley}.) Ignoring that, we
can still observe that the second moment of $\mathsf{invol}$ is quite obstreperous;
its growth rate vastly outpaces the first moment, and the extent to which it does depends
on whether $\theta > 1$.

For a hint at this disparity, note that the Ewens Sampling Formula favors permutations with many cycles whenever $\theta > 1.$ This inflates the likelihood of drawing a permutation with a lot of short cycles, which is when $\mathsf{invol}(\sigma)$ is at its biggest. So $\E_n\mathsf{invol}^2$ grows drastically faster than it does in the $\theta < 1$ case and falls under an entirely different asymptotic regime.

The ensuing analysis is included here not only to contextualize
the sheer futility of the method of moments in this situation but also to help cast light on
the statistical spread of $\mathsf{invol}$ for $\theta$-weighted permutations
and how $\theta = 1$ acts as a ``phase transition point'' for the variance.
Reprising the notation presented in the proof of Theorem \ref{average}, we consider the generating function
\begin{align}
G(z) &= 1 + \sum_{n = 1}^{\infty} (\E_n\mathsf{invol}^2)\frac{\theta^{(n)}}{n!}z^n
\\ &=1 + \sum_{n = 1}^{\infty} \left(\sum_{\lambda \vdash n} \PP_{n}(\sigma \in \mathcal{A}_{\lambda})\E_{n}\!\left(\mathsf{invol}^2\, |\, \sigma \in \mathcal{A}_{\lambda}\right)\right)\frac{\theta^{(n)}}{n!}z^n
\\ &= 1 + \sum_{n = 1}^{\infty} \left(\sum_{\lambda \vdash n} \left(\frac{n!}{\theta^{(n)}} \prod_{k = 1}^{n} \left(\frac{\theta}{k}\right)^{\lambda_k}\frac{1}{\lambda_k!}\right)\left(\prod_{k = 1}^{n} k^{2\lambda_k}\frac{\left(He_{\lambda_k}(i\sqrt{k})\right)^2}{(i\sqrt{k})^{2\lambda_k}}\right)\right)\frac{\theta^{(n)}}{n!}z^n
\\ &= \prod_{k = 1}^{\infty}\left(\sum_{\lambda_k = 0}^{\infty} \frac{\left(He_{\lambda_k}(i\sqrt{k})\right)^2}{\lambda_k!}\cdot(-\theta z^k)^{\lambda_k}\right)
\end{align}
We now recall the probabilists' version of the Mehler kernel (see \cite{Kibble})
\[\sum_{n = 0}^{\infty} He_n(x)He_n(y)\frac{t^n}{n!} = \frac{1}{\sqrt{1 - t^2}}\exp\!\left(-\frac{t^2(x^2 + y^2) - 2txy}{2(1 - t^2)}\right)\]
to observe that
\begin{equation}
\label{secondmoment}
G(z) = \prod_{k = 1}^{\infty} \frac{1}{\sqrt{1 - \theta^2z^{2k}}}\exp\!\left(\frac{\theta kz^k}{1 - \theta z^k}\right).
\end{equation}

\begin{theorem}
If $\theta > 1,$ then
\[\E_n\mathsf{invol}^2 = \frac{K_{\theta}\Gamma(\theta)}{2\sqrt{2\pi}}\theta^ne^{2\sqrt{n} - \frac{1}{2}}n^{\frac{1}{2} - \theta}\!\left(1 + O(n^{-\frac{1}{2}})\right).\]
\end{theorem}

\begin{proof}
Note that $G$ has a dominant singularity at $z = 1/\theta$ and the function $\phi$ given by
\begin{equation}
\phi(z) = \frac{\exp(-1)}{\sqrt{1 + \theta z}}\prod_{k = 2}^{\infty} \frac{1}{\sqrt{1 - \theta^2z^{2k}}}\exp\!\left(\frac{\theta kz^k}{1 - \theta z^k}\right)
\end{equation}
is regular in the disc $|z| < 1/\theta.$ Corollary \ref{Wrightasymptotic} can thus be applied to get
\begin{equation}
[z^n]G(z) = [z^n](1-\theta z)^{-\frac{1}{2}}\phi(z)\exp\!\left(\frac{1}{1-\theta z}\right) = \frac{\theta^n}{n^{\frac{1}{2}}}\frac{K_{\theta}\exp(2\sqrt{n} - 1/2)}{2\sqrt{2\pi}}\!\left(1 + O(n^{-\frac{1}{2}})\right)
\end{equation}
where
\begin{equation}
K_{\theta} = \prod_{k = 2}^{\infty} \frac{1}{\sqrt{1 - \theta^{2 - 2k}}}\exp\!\left(\frac{k}{\theta^{k-1} - 1}\right).
\end{equation}
Hence
\begin{equation}
\E_n\mathsf{invol}^2 = \frac{n!}{\theta^{(n)}}[z^n]G(z) = \frac{K_{\theta}\Gamma(\theta)}{2\sqrt{2\pi}}\theta^ne^{2\sqrt{n} - \frac{1}{2}}n^{\frac{1}{2} - \theta}\!\left(1 + O(n^{-\frac{1}{2}})\right).
\end{equation}
\end{proof}

Here we see that $\E_n(\mathsf{invol})^2 = \omega(\E_n\mathsf{invol}).$
However, if $0 < \theta \leq 1,$ then $G$ has a radius of convergence of 1,
but behaves too chaotically around the unit circle to be analytically continued beyond it.
In this case, we appeal to the exp-log reorganization of (\ref{secondmoment}):
\begin{align}
G(z) &= \exp\!\left(\sum_{k = 1}^{\infty} \left[-\frac{1}{2}\log(1 - \theta^2z^{2k}) + \frac{\theta kz^{k}}{1 - \theta z^k}\right]\right)
\\ &= \exp\!\left(\sum_{k = 1}^{\infty} \sum_{\ell = 1}^{\infty} \left[\frac{\theta^{2\ell}z^{2k\ell}}{2\ell} + k\theta^{\ell}z^{k\ell}\right]\right)
\\&= \exp\!\left(\sum_{\ell = 1}^{\infty} \left[\frac{\theta^{2\ell}}{2\ell}\frac{z^{2\ell}}{(1 - z^{2\ell})} + \frac{\theta^{\ell}z^{\ell}}{(1 - z^{\ell})^2}\right]\right).
\end{align}
From this, we can attain the following upper bound on $\left|\log G(z)\right|$ for when $0 < |z| < 1:$

\begin{align}
\left|\log G(z)\right| &\leq \frac{\theta|z|}{|1 - z|^2} + \sum_{\ell = 1}^{\infty} \frac{\theta^{2\ell}}{2\ell}\frac{|z|^{2\ell}}{|1 - z^{2\ell}|} + \sum_{\ell = 2}^{\infty} \frac{\theta^{\ell}|z|^{\ell}}{|1 - z^{\ell}|^2}
\\ &< \frac{\theta}{|1 - z|^2} + \sum_{\ell = 1}^{\infty} \frac{\theta^{2\ell}}{(2\ell)^2}\frac{2\ell|z|^{2\ell}}{(1 - |z|^{2\ell})} + \sum_{\ell = 2}^{\infty} \frac{\theta^{\ell}|z|^{\ell}}{(1 - |z|^{\ell})^2}
\\ &< \frac{\theta}{|1 - z|^2} + \frac{1}{1 - |z|}\sum_{\ell = 1}^{\infty} \frac{\theta^{2\ell}}{(2\ell)^2}\frac{2\ell}{(|z|^{-1} + |z|^{-2} + \cdots + |z|^{-2\ell})} \nonumber
\\ & \hspace{2.5em} + \frac{1}{(1 - |z|)^2}\sum_{\ell = 2}^{\infty} \frac{\theta^{\ell}}{(1 + |z| + \cdots + |z|^{\ell - 1})(1 + |z|^{-1} + \cdots + |z|^{-(\ell - 1)})} \label{AM}
\\ &< \frac{\theta}{|1 - z|^2} + \frac{1}{1 - |z|}\sum_{\ell = 1}^{\infty} \frac{\theta^{2\ell}}{4\ell^2} + \frac{1}{(1 - |z|)^2}\sum_{\ell = 2}^{\infty} \frac{\theta^{\ell}}{\ell^2} \label{HM}
\\ &= \frac{\theta}{|1 - z|^2} + \frac{\text{Li}_2(\theta^2)}{4(1 - |z|)} + \frac{\text{Li}_2(\theta) - \theta}{(1 - |z|)^2}, \label{estimate}
\end{align}
where $\text{Li}_s(z) = \sum_{n \geq 1} \frac{z^n}{n^s}$ is the polylogarithm function of order $s.$ (Apply the arithmetic mean-harmonic mean inequality on the terms of the second summation to transition from (\ref{AM}) to (\ref{HM}).)

Now in setting $z = e^{-u},$ the function
\[L(u) = \log G(e^{-u}) = \sum_{\ell = 1}^{\infty} \left[\frac{\theta^{2\ell}}{2\ell}\frac{e^{-2\ell u}}{(1 - e^{-2\ell u})} + \frac{\theta^{\ell}e^{-\ell u}}{(1 - e^{-\ell u})^2}\right]\]
is a harmonic sum well-suited for a Mellin transform analysis.
Consult section \ref{sec:mellin} of the appendix to deduce that the Mellin transform of $L(u)$ is
\[L^{\star}(s) = \left(2^{-s - 1}\text{Li}_{s + 1}(\theta^2)\zeta(s) + \text{Li}_s(\theta)\zeta(s - 1)\right)\!\Gamma(s).\]
Note that $\text{Li}_s(z),$ when treated as a function of $s$ and with $z$ held constant, is analytic for $z \in [0, 1),$ because $\text{Li}_s(z)$ converges absolutely for all $s \in \C$ when $|z| < 1.$ If $z = 1,$ then the polylogarithm is simply the Riemann zeta function which has a simple pole at $s = 1.$ We therefore consider the cases $0 < \theta < 1$ and $\theta = 1$ separately.

\begin{theorem}
If $0 < \theta < 1,$ then, as $n \to \infty,$
\[\E_n\mathsf{invol}^2 \sim \frac{(1 - \theta^2)^{\frac{1}{4}}(2\textup{Li}_2(\theta))^{\frac{1}{6}}\Gamma(\theta)n^{\frac{1}{3} - \theta}}{\sqrt{6\pi}}\exp\!\left(\frac{3\textup{Li}_2(\theta)n^{\frac{2}{3}}}{(2\textup{Li}_2(\theta))^{\frac{2}{3}}} + \frac{\textup{Li}_2(\theta^2)n^{\frac{1}{3}}}{(2\textup{Li}_2(\theta))^{\frac{1}{3}}} - \frac{\theta}{12(1 - \theta)}\right).\]
\end{theorem}

\begin{proof}
The Mellin transform $L^{\star}(s)$ has simple poles at $s = 2, 1, 0, -1, \ldots$ Using the standard transfer rule for the transform, we ascertain that
\begin{equation}
\label{Mellin}
L(t) \underset{t \to 0^+}{=} \frac{\text{Li}_2(\theta)}{t^2} + \frac{\text{Li}_2(\theta^2)}{4t} + \frac{1}{4}\log(1 - \theta^2) - \frac{\theta}{12(1 - \theta)} + O(t)
\end{equation}
which, because of the Laurent series expansions of the powers of $t = -\log(z)$ at $z = 1$ and the fact that (\ref{Mellin}) extends to $\C - \R_{< 0},$ corresponds to
\begin{equation}
\log G(z) \underset{z \to 1}{=} \frac{\text{Li}_2(\theta)}{(1 - z)^2} + \frac{\text{Li}_2(\theta^2) - 4\text{Li}_2(\theta)}{4(1 - z)} + \frac{1}{12}\text{Li}_2(\theta) - \frac{1}{8}\text{Li}_2(\theta^2) +  \frac{1}{4}\log(1 - \theta^2) - \frac{\theta}{12(1 - \theta)} + O(1 - z)
\end{equation}
from within the unit disc. We can now conclude that
\begin{equation}
\label{bigG}
G(z) \underset{z \to 1}{=} \tau(z)[1 + O(1 - z)],
\end{equation}
provided $|z| < 1$ with
\begin{equation}
\tau(z) = (1 - \theta^2)^{1/4}\exp\!\left(\frac{\text{Li}_2(\theta)}{(1 - z)^2} + \frac{\text{Li}_2(\theta^2) - 4\text{Li}_2(\theta)}{4(1 - z)} + \frac{1}{12}\text{Li}_2(\theta) - \frac{1}{8}\text{Li}_2(\theta^2) - \frac{\theta}{12(1 - \theta)}\right)
\end{equation}
We claim next that $[z^n]G(z) = [z^n]\tau(z)(1 + o(1)).$ To see this, consider $\gamma,$ the positively oriented circle $|z| = 1 - \left(2\text{Li}_2(\theta)/n\right)^{1/3},$ and split $\gamma$ into the two contours
\[\gamma_1 = \left\{z \in \gamma : |1 - z| < \left(\frac{3\text{Li}_2(\theta)}{n}\right)^{\frac{1}{3}}\right\}\]
and $\gamma_2 = \gamma - \gamma_1.$ By Cauchy's integral formula,
\begin{equation}
\label{cauchy}
[z^n](G(z) - \tau(z)) = \frac{1}{2\pi i}\int_{\gamma_1} \frac{G(z) - \tau(z)}{z^{n + 1}}\,dz + \frac{1}{2\pi i}\int_{\gamma_2} \frac{G(z) - \tau(z)}{z^{n + 1}}\,dz.
\end{equation}
The length of $\gamma_1$ is $O(n^{-1/3}).$ This together with (\ref{bigG}) shows that the first integral over $\gamma_1$ is
\begin{align}
&= O\!\left(\int_{\gamma_1} |z|^{-n}|1 - z|\exp\!\left(\frac{\text{Li}_2(\theta)}{(1 - |z|)^2} + \frac{\text{Li}_2(\theta^2) - 4\text{Li}_2(\theta)}{4(1 - |z|)}\right)|dz|\right)
\\ &= O\!\left(\exp\!\left(\frac{2\text{Li}_2(\theta)n^{\frac{2}{3}}}{(2\text{Li}_2(\theta))^{\frac{2}{3}}}\right)\!\cdot\!n^{-\frac{1}{3}}\!\cdot\!\exp\!\left(\frac{\text{Li}_2(\theta)n^{\frac{2}{3}}}{(2\text{Li}_2(\theta))^{\frac{2}{3}}} + \frac{(\text{Li}_2(\theta^2) - 4\text{Li}_2(\theta))n^{\frac{1}{3}}}{4(2\text{Li}_2(\theta))^{\frac{2}{3}}}\right)\right)\!\cdot\!O(n^{-\frac{1}{3}})
\\ &= O\!\left(n^{-\frac{2}{3}}\exp\!\left(\frac{3\text{Li}_2(\theta)n^{\frac{2}{3}}}{(2\text{Li}_2(\theta))^{\frac{2}{3}}} + \frac{(\text{Li}_2(\theta^2) - 4\text{Li}_2(\theta))n^{\frac{1}{3}}}{4(2\text{Li}_2(\theta))^{\frac{1}{3}}}\right)\right)
\end{align}
On the other hand, estimate (\ref{estimate}) applies over $\gamma_2,$ and so the second integral in (\ref{cauchy}) is
\begin{align}
&O\!\left(\int_{\gamma_2} |z|^{-n}\!\left[\exp\!\left(\frac{\theta}{|1 - z|^2} + \frac{\text{Li}_2(\theta^2)}{4(1 - |z|)} + \frac{\text{Li}_2(\theta) - \theta}{(1 - |z|)^2}\right) + \exp\!\left(\frac{\text{Li}_2(\theta)}{|1 - z|^2} + \frac{\text{Li}_2(\theta^2) - 4\text{Li}_2(\theta)}{4|1 - z|}\right)\!\right]\!|dz|\!\right)
\\ &= O\!\left(\exp\!\left(\frac{2\text{Li}_2(\theta)n^{\frac{2}{3}}}{(2\text{Li}_2(\theta))^{\frac{2}{3}}}\right)\cdot\exp\!\left(\frac{\theta n^{\frac{2}{3}}}{(3\text{Li}_2(\theta))^{\frac{2}{3}}} + \frac{\text{Li}_2(\theta^2)n^{\frac{1}{3}}}{4(2\text{Li}_2(\theta))^{\frac{1}{3}}} + \frac{(\text{Li}_2(\theta) - \theta)n^{\frac{2}{3}}}{(2\text{Li}_2(\theta))^{\frac{2}{3}}}\right)\right. \nonumber
\\ & \qquad \hspace{2.5em} \left. + \exp\!\left(\frac{2\text{Li}_2(\theta)n^{\frac{2}{3}}}{(2\text{Li}_2(\theta))^{\frac{2}{3}}}\right)\cdot\exp\!\left(\frac{\text{Li}_2(\theta)n^{\frac{2}{3}}}{(3\text{Li}_2(\theta))^{\frac{2}{3}}} + \frac{(\text{Li}_2(\theta^2) - 4\text{Li}_2(\theta))n^{\frac{1}{3}}}{4(3\text{Li}_2(\theta))^{\frac{1}{3}}}\right)\right)
\\ &= O\!\left(\exp(An^{\frac{2}{3}})\right),
\end{align}
for some $A < 3\text{Li}_2(\theta)/(2\text{Li}_2(\theta))^{2/3}.$ Therefore
\begin{equation}
\label{error}
[z^n]G(z) = [z^n]\tau(z) + O\!\left(n^{-\frac{2}{3}}\exp\!\left(\frac{3\text{Li}_2(\theta)n^{\frac{2}{3}}}{(2\text{Li}_2(\theta))^{\frac{2}{3}}} + \frac{(\text{Li}_2(\theta^2) - 4\text{Li}_2(\theta))n^{\frac{1}{3}}}{4(2\text{Li}_2(\theta))^{\frac{1}{3}}}\right)\right)
\end{equation}
It is now a sophomore's homework exercise to use Theorem \ref{Wrightgeneral} on $\tau(z),$ especially considering that the required singular expansion is already accounted for in (\ref{Mellin}). The error term in (\ref{error}) is smaller than the leading term asymptotic of $[z^n]\tau(z)$ by a decaying factor of $O(\exp(-\text{Li}_2(\theta)n^{1/3}/(2\text{Li}_2(\theta))^{1/3})).$ In the end, we find that
\begin{equation}
[z^n]G(z) \sim \frac{(1 - \theta^2)^{\frac{1}{4}}(2\text{Li}_2(\theta))^{\frac{1}{6}}}{\sqrt{6\pi}n^{\frac{2}{3}}}\exp\!\left(\frac{3\text{Li}_2(\theta)n^{\frac{2}{3}}}{(2\text{Li}_2(\theta))^{\frac{2}{3}}} + \frac{\text{Li}_2(\theta^2)n^{\frac{1}{3}}}{(2\text{Li}_2(\theta))^{\frac{1}{3}}} - \frac{\theta}{12(1 - \theta)}\right),
\end{equation}
and so the result follows again from the fact that $\E_n\mathsf{invol}^2 = \frac{n!}{\theta^{(n)}}[z^n]G(z).$
\end{proof}

\begin{theorem}
If $\theta = 1,$ then
\[\E_n\mathsf{invol}^2 \sim \frac{n!}{\sqrt[4]{2}\sqrt[12]{\pi}(3n)^{\frac{7}{12}}}\exp\!\left(\frac{1}{2}(3\pi n)^{\frac{2}{3}} - \frac{\pi^2 - 4}{8}\!\left(\frac{3n}{\pi^2}\right)^{\frac{1}{3}} + \frac{7}{24} - \frac{\pi^2}{144}\right).\]
\end{theorem}

\begin{proof}
Here, $G(z)$ is precisely the exponential generating function of $\E_n\mathsf{invol}^2,$ and the Mellin transform becomes
\begin{equation}
L^{\star}(s) = (2^{-s - 1}\zeta(s + 1) + \zeta(s - 1))\zeta(s)\Gamma(s),
\end{equation}
which has simple poles at $s = 2, 1, 0, -1, \ldots,$ and a double pole at $s = 0.$ Hence
\begin{equation}
L(t) \underset{t \to 0^+}{=} \frac{\pi^2}{6t^2} + \frac{\pi^2 - 12}{24t} + \frac{1}{4}\log t - \log\sqrt[4]{2\pi} + \frac{1}{24} + O(t),
\end{equation}
which corresponds to
\begin{equation}
\log G(z) \underset{z \to 1}{=} \frac{\pi^2}{6(1 - z)^2} - \frac{\pi^2 - 4}{8(1 - z)} + \frac{1}{4}\log(1 - z) + \frac{7}{24} - \frac{\pi^2}{144} - \log\sqrt[4]{2\pi} + O(1 - z)
\end{equation}
from within the unit disc. Therefore
\begin{equation}
G(z) = \tau(z)[1 + O(1 - z)],
\end{equation}
provided $|z| < 1$ and
\begin{equation}
\tau(z) = (2\pi(1 - z))^{\frac{1}{4}}\exp\!\left(\frac{\pi^2}{6(1 - z)^2} - \frac{\pi^2 - 4}{8(1 - z)} + \frac{7}{24} - \frac{\pi^2}{144}\right).
\end{equation}
At this point, the outstanding justifications are practically the same as they were for the $\theta \in (0, 1)$ case. It is tedious, but requires nothing new from earlier to figure out that
\begin{equation}
[z^n]G(z) \sim \frac{1}{\sqrt[4]{2}\sqrt[12]{\pi}(3n)^{\frac{7}{12}}}\exp\!\left(\frac{1}{2}(3\pi n)^{\frac{2}{3}} - \frac{\pi^2 - 4}{8}\!\left(\frac{3n}{\pi^2}\right)^{\frac{1}{3}} + \frac{7}{24} - \frac{\pi^2}{144}\right).
\end{equation}
\end{proof}

\section{A Central Limit Theorem for $\log\mathsf{invol}$}
\label{sec:clt}
In this section, we emulate the proof of Theorem 4.3 in \cite{BS}, except now each
$\mathfrak{S}_n$ is equipped with the ESF$(\theta)$ probability measure $\PP_n$
with fixed positive $\theta$. We proceed in three steps. First, $\mathsf{invol}(\sigma)$
is well-approximated by $\B(\sigma)$ for permutations $\sigma$ in which
only $k$-cycles with $k$ under a certain threshold $\xi,$ to be stipulated \textit{a posteriori},
can occur more than once but still not more than $\xi$ times. Call the set of all such
permutations $\mathcal{P}_{\xi}.$ The second step is to show that
$\PP_n(\sigma \in \mathcal{P}_{\xi}) \to 1$ as $n \to \infty.$ Lastly, the
previous two steps are combined to compare the distribution of $\log\mathsf{invol}$
to the distribution of $\log\B.$

There are two notable improvements here over the presentation of \cite{Burnette} and \cite{BS}.
Our present proof of the aforementioned second step is cleaner and considerably streamlined.
Furthermore, a concrete error term for the approximate Gaussian law will be added in the next section. Perhaps
unsurprisingly, the error term matches the convergence rate of the Erd\H{o}s-Tur\'{a}n limit law for
the order of a random permutation that was first determined in \cite{BT}.

\begin{lemma}  
\label{deterministic}
\textup{(\cite{BS})}
Suppose $\sigma \in \mathcal{P}_{\xi}.$ Then there is a constant $c > 0,$ not dependent on $\sigma$ nor $\xi,$ such that $\B(\sigma) \leq \mathsf{invol}(\sigma) \leq   \B(\sigma) \cdot \left( c\xi^{\xi}\right)^{\xi}.$
\end{lemma}
Lemma \ref{deterministic} is entirely deterministic and thus needs not be revisited. The next lemma generalizes Lemma III in \cite{ET1}.

\begin{lemma}
\label{probabilistic}
If $\xi = \xi(n) \to \infty,$ then $\PP_n(\sigma \not\in \mathcal{P}_{\xi}) = O(\frac{1}{\xi}).$
\end{lemma}

\begin{proof}
Due to (\ref{Ewensdistribution}) and the exponential formula for labeled combinatorial structures,
\begin{equation}
\label{typicalgf}
\PP_n(\sigma \in \mathcal{P}_{\xi}) = \frac{n!}{\theta^{(n)}}[z^n]\!\left[\left(\prod_{k = 1}^{\lfloor\xi\rfloor} \sum_{j = 1}^{\lfloor\xi\rfloor}\left(\frac{\theta z^k}{k}\right)^{j}\frac{1}{j!}\right)\!\cdot\!\left(\prod_{k = \lfloor\xi\rfloor + 1}^{\infty}\left(1 + \frac{\theta z^k}{k}\right)\right)\right].
\end{equation}
The product in (\ref{typicalgf}) can be written as $(1 - z)^{-\theta}\Omega(z),$ where
\begin{equation}
\Omega(z) = \left(\prod_{k = 1}^{\lfloor\xi\rfloor}\left(1 - e^{-\frac{\theta z^k}{k}}\sum_{j = \lfloor\xi\rfloor + 1}^{\infty} \left(\frac{\theta z^k}{k}\right)^{j}\frac{1}{j!}\right)\right)\!\cdot\!\left(\prod_{k = \lfloor\xi\rfloor + 1}^{\infty}\left(1 - e^{-\frac{\theta z^k}{k}}\sum_{j = 2}^{\infty} \left(\frac{\theta z^k}{k}\right)^{j}\frac{1}{j!}\right)\right).
\end{equation}
Equating the corresponding coefficients reveals that
\begin{equation}
\PP_n(\sigma \in \mathcal{P}_{\xi}) - 1 = \frac{n!}{\theta^{(n)}}\sum_{\ell = 1}^{n} \binom{\theta + n - \ell - 1}{n - \ell}[z^\ell]\Omega(z) \leq \sum_{\ell = 1}^{n} [z^{\ell}]\Omega(z).
\end{equation}
If we replace every instance of the term $-e^{-\frac{\theta z^k}{k}}$ by $e^{\frac{\theta z^k}{k}}$ in each factor of $\Omega(z),$ we obtain an entire function $\Omega^*(z)$ whose coefficients are all positive and majorize the absolute values of the corresponding coefficients of $\Omega(z).$ Hence
\begin{align}
\PP_n(\sigma \not\in \mathcal{P}_{\xi}) &= 1 - \PP_n(\sigma \in \mathcal{P}_{\xi})
\\ &\leq \sum_{\ell = 1}^{n} [z^{\ell}]\Omega^*(z)
\\ &< \Omega^*(1) - \Omega^*(0)
\\ &= \left(\prod_{k = 1}^{\lfloor\xi\rfloor}\left(1 + e^{\frac{\theta}{k}}\sum_{j = \lfloor\xi\rfloor + 1}^{\infty} \left(\frac{\theta}{k}\right)^{j}\frac{1}{j!}\right)\right)\!\cdot\!\left(\prod_{k = \lfloor\xi\rfloor + 1}^{\infty}\left(1 + e^{\frac{\theta}{k}}\sum_{j = 2}^{\infty} \left(\frac{\theta}{k}\right)^{j}\frac{1}{j!}\right)\right) - 1.
\end{align}
For sufficiently large $n,$ we have
\begin{align}
\prod_{k = 1}^{\lfloor\xi\rfloor}\left(1 + e^{\frac{\theta}{k}}\sum_{j = \lfloor\xi\rfloor + 1}^{\infty} \left(\frac{\theta}{k}\right)^{j}\frac{1}{j!}\right) &< \exp\!\left(\sum_{k = 1}^{\lfloor\xi\rfloor}e^{\frac{\theta}{k}}\sum_{j = \lfloor\xi\rfloor + 1}^{\infty} \left(\frac{\theta}{k}\right)^{j}\frac{1}{j!}\right)
\\ &< \exp\!\left(\frac{e^{2\theta}\theta^{\lfloor\xi\rfloor + 1}}{(\lfloor\xi\rfloor + 1)!}\sum_{k = 1}^{\lfloor\xi\rfloor} \frac{1}{k^{\lfloor\xi\rfloor + 1}}\right)
\\ &< \exp\!\left(\frac{K_1\theta^{\lfloor\xi\rfloor + 1}}{(\lfloor\xi\rfloor + 1)!}\right)
\end{align}
for some constant $K_1,$ and
\begin{align}
\prod_{k = \lfloor\xi\rfloor + 1}^{\infty}\left(1 + e^{\frac{\theta}{k}}\sum_{j = 2}^{\infty} \left(\frac{\theta}{k}\right)^{j}\frac{1}{j!}\right) &< \exp\!\left(\sum_{k = \lfloor\xi\rfloor}^{\infty}e^{\frac{\theta}{k}}\sum_{j = 2}^{\infty} \left(\frac{\theta}{k}\right)^{j}\frac{1}{j!}\right)
\\ &< \exp\!\left(\frac{e^{\theta + 2}\theta^2}{2}\sum_{k = \lfloor\xi\rfloor}^{\infty} \frac{1}{k^2}\right)
\\ &< \exp\!\left(\frac{K_2}{\xi}\right)
\end{align}
for some constant $K_2.$ Therefore, as $n \to \infty,$
\begin{equation}
\PP_n(\sigma \not\in \mathcal{P}_{\xi}) < \exp\!\left(\frac{K_1\theta^{\lfloor\xi\rfloor + 1}}{(\lfloor\xi\rfloor + 1)!} + \frac{K_2}{\xi}\right) - 1 = O\!\left(\frac{1}{\xi}\right).
\end{equation}
\end{proof}

We now cite the asymptotic lognormality of $\B$ for $\theta$-weighted permutations.

\begin{lemma} 
\label{ET}
\textup{(\cite{ABT2})}
\[\sup_{x \in \R}\left|\mathbb{P}_{n}\!\left(\frac{\log\B(\sigma) - \mu_{n}}{\sigma_n}\leq x\right) - \Phi(x)\right| = O\!\left(\frac{1}{\sqrt{\log n}}\right),\]
where
$\mu_n :=\sum\limits_{k=1}^{n} \frac{\theta\log k}{k} \sim  \frac{\theta}{2}\log^2\!n$,
$\sigma^{2}_n :=  \sum\limits_{k=1}^{n} \frac{\theta\log^{2}\!k}{k}\sim $
$\frac{\theta}{3}\log^{3}\!n$,  and $ \Phi(x) := \frac{1}{\sqrt{2\pi}} $
$\int_{-\infty}^{x}e^{-t^2/2}\,dt .$
\end{lemma}

\cite{ET4} were the first to confirm the asymptotic lognormality of $\B$ for uniform random permutations. Progressively sharper versions of the Erd\H{o}s-Tur\'{a}n limit law for $\B$ have emerged on multiple occasions, e.g. \cite{Bovey}, \cite{DP}, \cite{Manstavicius}, \cite{Nicolas}, \cite{Zacharovas}. Proofs of Lemma \ref{ET} itself appear in the aggregate works of Arratia, Barbour, and Tavar\'{e} throughout \cite{ABT1}, \cite{ABT2}, \cite{ATdp}, \cite{BT}.

\begin{theorem}
\label{mainthm}
For all real $x,$
\[\lim_{n \to \infty} \mathbb{P}_{n}\!\left(\frac{\log\mathsf{invol}(\sigma) - \mu_{n}}{\sigma_n}\leq x\right) = \Phi(x)\]
\end{theorem}

\begin{proof}
Because $\mathsf{invol}(\sigma) \geq \B(\sigma)$ for all $\sigma \in \mathfrak{S}_n,$ 
an upper bound is immediate:
\begin{equation}
\mathbb{P}_{n}\!\left(\frac{\log\mathsf{invol}(\sigma) - \mu_{n}}{\sigma_n}\leq x\right) \leq 
\mathbb{P}_{n}\!\left(\frac{\log\B(\sigma) - \mu_{n}}{\sigma_n}\leq x\right) = \Phi(x) + o(1).
\end{equation}
For a lower bound, we capitalize on the uniform continuity of $\Phi.$ We also invoke the conditional bound
from Lemma \ref{deterministic}, where $\xi$ will now be chosen to our liking.
Due to Lemma \ref{probabilistic}, this bound holds with probability 
$1 - O(\frac{1}{\xi}).$

The remainder of the proof is virtually identical to that of Theorem 4.3 in \cite{BS}, which is included here for the sake of completeness. Let $\epsilon >0$ be a fixed but arbitrarily small positive number. 
We can choose $\delta>0$ so that $|\Phi(x) - \Phi(a)| < \epsilon$ 
whenever $|x - a| < \delta.$
If we choose $\xi=\sqrt{\log n},$  then we have
$ \log\!\left( (c\xi^{\xi} )^{\xi}\right)=o(\sigma_{n}).$
Therefore we can choose  $N_{\epsilon}$ so that, for all $n\geq N_{\epsilon}$, we have
$  \log\!\left( (c\xi^{\xi} )^{\xi}\right)< \frac{\delta \sigma_{n}}{2}.$
It follows that, uniformly in $x,$
\begin{align}
\mathbb{P}_{n}\!\left(\frac{\log\mathsf{invol}(\sigma) - \mu_{n}}{\sigma_n}\leq x\right) &\geq 
\mathbb{P}_n\!\left(\frac{\log\B(\sigma) +  \log\!\left( (c\xi^{\xi} )^{\xi}\right) - \mu_n}{\sigma_n}
\leq x\right) + o(1)
\\ &\geq \mathbb{P}_n\!\left(\frac{\log\B(\sigma) +  \delta\sigma_n/2 - \mu_n}{\sigma_n}\leq x\right) + o(1)
\\ &= \mathbb{P}_n\!\left(\frac{\log\B(\sigma) - \mu_n}{\sigma_n}\leq x - \frac{\delta}{2}\right) + o(1)
\\ &= \Phi\!\left(x - \frac{\delta}{2}\right) + o(1) > \Phi(x) - \epsilon + o(1).
\end{align}
Yet $\epsilon > 0$ was arbitrary, and so $\mathbb{P}_{n}\!\left(\frac{\log\mathsf{invol}(\sigma) - \mu_{n}}{\sigma_n}\leq x\right) \geq \Phi(x) + o(1).$ 
\end{proof}

\section{A Functional Refinement and Convergence Rate}
\label{sec:refine}
A stochastic process variant of Theorem \ref{mainthm} can also be formulated. Let $C^{(n)} = (c^{(n)}_1, c^{(n)}_2, \ldots)$ be distributed according to $\text{ESF}(\theta)$ and $W_n$ be the random element of the Skorokhod space $\mathcal{D}[0, 1]$ of c\`{a}dl\`{a}g functions on $[0, 1]$ defined by
\begin{equation}
W_n(t) = \frac{\log\mathsf{invol}_{\lfloor n^t \rfloor}(C^{(n)}) - \frac{\theta t^2}{2}\log^2\!n}{\sqrt{\frac{\theta}{3}\log^3\!n}}.
\end{equation}
\begin{theorem}
\label{functional}
It is possible to construct $C^{(n)}$ and a standard Brownian motion $W$ on the same probability space in such a way that
\[\E\!\left\{\sup_{0 \leq t \leq 1} \left|W_n(t) - W(t^3)\right|\right\} = O\!\left(\frac{\log\log n}{\sqrt{\log n}}\right).\]
\end{theorem}

The Feller coupling, which expresses the Ewens Sampling Formula in terms of
the spacings between successes in a sequence of independent Bernoulli trials,
offers the ideal terrain for this task. \cite{ABTsurvey}
wrote a brief expository note about the Feller coupling to act as a companion to
\cite{Crane}. To summarize, consider a sequence $\beta = (\beta_1, \beta_2, \ldots)$
of independent Bernoulli random variables in which $\text{Pr}(\beta_j = 1) = \theta/(\theta + j - 1),$ $j = 1, 2, \ldots$
Say that a $k$-spacing occurs in a given binary sequence $\beta$, starting at position $\ell - k$
and ending at position $\ell,$ if $\beta_{\ell - k} = \beta_{\ell} = 1$ and
$\beta_j = 0$ for $\ell - k + 1 \leq j \leq \ell-1.$ If for all positive integers $k$ and $n$ we define
\begin{equation}
c^{(n)}_k = \# k\text{-spacings in}\ 1\beta_2\beta_3\cdots\beta_n1,
\end{equation}
it can be shown that the distribution of $C^{(n)} = (c^{(n)}_1, c^{(n)}_2, \ldots)$ is exactly the same as it is in $\text{ESF}(\theta).$ If we further define $Z_m = (Z_{1, m}, Z_{2, m}, \ldots),$ where
\begin{equation}
Z_{k, m} = \# k\text{-spacings in}\ \beta_{m + 1}\beta_{m + 2}\cdots, \hspace{2.5em} k, m \geq 1,
\end{equation}
then $Z_m \overset{\text{a.s.}}{\longrightarrow} 0$ as $m \to \infty$ due to the Borel-Cantelli lemma, and so $Z_0$ is equivalent to the Poisson point process to which $C^{(n)}$ itself converges.

We now assemble the pieces needed to prove Theorem \ref{functional}. Our plan is to show that within the Feller coupling, $\log\mathsf{invol}_n(C^{(n)})$ is ``close enough'' to $\log\mathsf{invol}_n(Z_0),$ which in turn is ``close enough'' to $\log\B_n(Z_0),$ which is known to conform to a Wiener process. To this end, we raise two lemmas. The first is a trivial consequence of the preceding definitions.

\begin{lemma}
\label{trivial}
For all positive integers $k$ and $n,$
\[Z_{k,0} - Z_{k, n} - \mathbbold{1}\{L_n + R_n = k + 1\} \leq c^{(n)}_k \leq Z_{k,0} + \mathbbold{1}\{L_n = k\},\]
where $L_j = \min\{j \geq 1 : \beta_{n - j + 1} = 1\}$ and $R_j = \min\{j \geq 1 : \beta_{n + j} = 1\}.$
\end{lemma}

\begin{lemma}
\label{stochastic}
If $a$ and $b$ are two vectors such that $0 \leq a_k \leq b_k$ for $k \in [n],$ then
\[\mathsf{invol}_{j}(b) \leq \mathsf{invol}_j(a)n^{\|b - a\|_1}\prod_{k = 1}^{n}\exp\!\left(\frac{b_k^2}{2k}\right),\]
for all $j \in [n],$ where $\|b-a\|_1$ denotes the usual $1$-norm.
\end{lemma}

\begin{proof}
Observe that
\begin{equation}
\mathsf{invol}_j(b) = \mathsf{invol}_j(a)\left(\prod_{k = 1}^{j} k^{b_k - a_k}\right)\cdot\left(\prod_{k = 1}^{j} \frac{V_{b_k}(k)}{V_{a_k}(k)}\right),
\end{equation}
where we have set
\[V_m(k) = \sum\limits_{\ell=0}^{\lfloor m/2\rfloor} \frac{(m)_{2\ell}}{(2k)^{\ell}\ell!}.\]
Note that $1 \leq V_m(k) \leq \exp(\frac{m^2}{2k})$ for all positive integers $m$ and $k.$
\end{proof}

\begin{lemma}
\[\E\!\left\{\sup_{0 \leq t \leq 1} \left|\log\mathsf{invol}_{\lfloor n^t \rfloor}(C^{(n)}) - \log\B_{\lfloor n^t \rfloor}(Z_0)\right|\right\} = O(\log n).\]
\end{lemma}

\begin{proof}
Set $Y_n := \sum_{k = 1}^{n} Z_{k, n},$ $D_n := \sum_{k = 1}^{n} \frac{(c^{(n)}_k + Z_{k, n} + 1)^2}{2k},$ and $D^{*}_n := \sum_{k = 1}^{n} \frac{(Z_{k,0} + 1)^2}{2k}.$ It follows from Lemmas \ref{trivial} and \ref{stochastic} that
\begin{align}
\log\mathsf{invol}_j(Z_0) &\leq \log\mathsf{invol}_j(C^{(n)} + Z_n + (\mathbbold{1}\{L_n + R_n = 2\}, \ldots, \mathbbold{1}\{L_n + R_n = n + 1\}))
\\ &\leq \log\mathsf{invol}_j(C^{(n)}) + (Y_n + 1)\log n + D_n
\end{align}
and
\begin{align}
\log\mathsf{invol}_j(C^{(n)}) &\leq \log\mathsf{invol}_j(Z_0 + (\mathbbold{1}\{L_n = 1\}, \ldots, \mathbbold{1}\{L_n = n\}))
\\ &\leq \log\mathsf{invol}_j(Z_0) + \log n + D^{*}_n,
\end{align}
for all $j \in [n].$ By noticing that
\begin{equation}
D_n \leq \sum_{k = 1}^{n} \frac{(2Z_{k,0}+2)^2}{2k} = 4D^*_n,
\end{equation}
we see that
\begin{equation}
\label{Ybound}
|\log\mathsf{invol}_{\lfloor n^t \rfloor}(C^{(n)}) - \log\mathsf{invol}_{\lfloor n^t \rfloor}(Z_0)| \leq (Y_n + 1)\log n + 4D^{*}_n
\end{equation}
for $0 \leq t \leq 1.$ \cite{BT} proved that $\E Y_n \leq \theta^2.$ Additionally,
\begin{equation}
\label{distance1}
\E D^{*}_n = \sum_{k = 1}^{n} \frac{1}{2k}\!\left(\frac{\theta^2}{k^2} + \frac{3\theta}{k} + 1\right) = O(\log n).
\end{equation}
Hence, 
\begin{equation}
\label{distance2}
\E\!\left\{\sup_{0 \leq t \leq 1} \left|\log\mathsf{invol}_{\lfloor n^t \rfloor}(C^{(n)}) - \log\mathsf{invol}_{\lfloor n^t \rfloor}(Z_0)\right|\right\} = O(\log n).
\end{equation}
Due to (1.5) and (1.6),
\begin{equation}
\E\!\left\{\sup_{0 \leq t \leq 1} \left|\log\mathsf{invol}_{\lfloor n^t \rfloor}(Z_0) - \log\B_{\lfloor n^t \rfloor}(Z_0)\right|\right\} = O(1).
\end{equation}
Triangle inequality implies the desired result.
\end{proof}

The remainder of the proof of Theorem \ref{functional} is immediate in light of the following.

\begin{theorem}
\label{brownian}
\textup{(\cite{BT})}
Suppose for all $t \in [0, 1]$ we define
\[B_n(t) = \frac{\log\B_{\lfloor n^t \rfloor}(Z_0) - \frac{\theta t^2}{2}\log^2\!n}{\sqrt{\frac{\theta}{3}\log^3\!n}}.\]
Then $B_n$ converges weakly in $\mathcal{D}[0, 1]$ to a standard Brownian motion $B,$ in the sense that
\[\E\!\left\{\sup_{0 \leq t \leq 1} \left|B_n(t) - B(t^3)\right|\right\} = O\!\left(\frac{\log\log n}{\sqrt{\log n}}\right).\]
\end{theorem}

We now articulate the proximity of the quantities $\log\mathsf{invol}_n(C^{(n)})$ and $\log\B_n(C^{(n)})$ with the goal of appending a convergence rate to Theorem \ref{mainthm}. Let us introduce the two standardized random variables
\[S_{1,n} = \frac{\log\B_n(C^{(n)}) - \frac{\theta}{2}\log^2\!n}{\sqrt{\frac{\theta}{3}\log^3\!n}} \hspace{2.5em} \text{and} \hspace{2.5em} S_{2,n} =  \frac{\log\mathsf{invol}_n(C^{(n)}) - \frac{\theta}{2}\log^2\!n}{\sqrt{\frac{\theta}{3}\log^3\!n}}.\]
The following generalization of Lemma 2.5 in \cite{BT} will help us show that the distributions of $S_{1, n}$ and $S_{2, n}$ are almost the same.

\begin{lemma}
\label{closeone}
Let $U$ and $X$ be random variables. Suppose that
\[\sup_{x \in \R} |\PP(U \leq x) - G(x)| \leq \eta,\]
where $G(x)$ is a Lipschitz function over $\R.$ Then, for any $\epsilon > 0,$
\[\sup_{x \in \R} |\PP(U + X \leq x) - G(x)| \leq \eta + C_G\epsilon + \PP(|X| > \epsilon)\]
where the constant $C_G$ depends only on $\|G'\|_{\infty}.$
\end{lemma}

\begin{theorem}
\label{errorrate}
If $C^{(n)}$ is distributed according to $\textup{ESF}(\theta),$ then
\[\sup_{x \in \R} \left|\PP\!\left(\frac{\log\mathsf{invol}_n(C^{(n)}) - \frac{\theta}{2}\log^2\!n}{\sqrt{\frac{\theta}{3}\log^3\!n}} \leq x\right) - \Phi(x)\right| = O\!\left(\frac{1}{\sqrt{\log n}}\right).\]
\end{theorem}
\begin{proof}
We begin by applying Lemma \ref{closeone} with $U = S_{1, n},$ and $X = S_{2, n} - S_{1, n}.$ By Lemma \ref{ET}, there exists a constant $c = c(\theta)$ such that $\sup_x |\PP(S_{1, n} \leq x) - \Phi(x)| \leq c\log^{-1/2}\!n.$ So if we take $\eta = c\log^{-1/2}\!n$ and $\epsilon = \log^{-1/2}\!n,$ then applying a generic Chernoff bound yields
\begin{align}
\PP(|S_{2, n} - S_{1, n}| > \epsilon) &= \PP\!\left(\log\mathsf{invol}_n(C^{(n)}) - \log\B_n(C^{(n)}) > \epsilon\sqrt{\frac{\theta}{3}\log^3\!n}\right)
\\ &\leq \frac{\E e^{\log\mathsf{invol}_n(C^{(n)}) - \log\B_n(C^{(n)})}}{e^{\epsilon\sqrt{\frac{\theta}{3}\log^3\!n}}}
\\ &\leq n^{-\sqrt{\frac{\theta}{3}}}\E\!\left[\prod_{k = 1}^{n}\sum\limits_{j=0}^{\lfloor c_k/2 \rfloor} \frac{(c^{(n)}_k)_{2j}}{(2k)^{j}j!}\right]
\\ &\leq n^{-\sqrt{\frac{\theta}{3}}}\E\!\left[\prod_{k = 1}^{n}\sum\limits_{j=0}^{\lfloor c_k/2 \rfloor} \frac{(Z_{k, 0} +\mathbbold{1}\{L_n = k\})_{2j}}{(2k)^{j}j!}\right].
\end{align}
Leverage the fact that $\mathbbold{1}\{L_n = m\}$ for at most one integer $m$ to see that
\begin{align}
\E\!\left[\prod_{k = 1}^{n}\sum\limits_{j=0}^{\lfloor c_k/2 \rfloor} \frac{(Z_{k, 0} +\mathbbold{1}\{L_n = k\})_{2j}}{(2k)^{j}j!}\right] &\leq \E\!\left[\left(\sum_{j = 0}^{\lfloor c_k/2 \rfloor} \frac{(Z_{m, 0} + 1)_{2j}}{(2m)^jj!}\right)\prod_{\substack{1 \leq k \leq n \\ k \neq m}}\sum\limits_{j=0}^{\infty} \frac{(Z_{k, 0})_{2j}}{(2k)^{j}j!}\right]
\\ &= \E\!\left[\left(\sum_{j = 0}^{\lfloor c_k/2 \rfloor} \frac{Z_{m, 0} + 1}{Z_{m, 0} - 2j + 1}\cdot\frac{(Z_{m, 0})_{2j}}{(2m)^jj!}\right)\!\prod_{\substack{1 \leq k \leq n \\ k \neq m}}\sum\limits_{j=0}^{\infty} \frac{(Z_{k, 0})_{2j}}{(2k)^{j}j!}\right]
\\ &\leq \E\!\left[\left(\sum_{j = 0}^{\infty} \frac{(1+2j)(Z_{m, 0})_{2j}}{(2m)^jj!}\right)\!\prod_{\substack{1 \leq k \leq n \\ k \neq m}}\sum\limits_{j=0}^{\infty} \frac{(Z_{k, 0})_{2j}}{(2k)^{j}j!}\right]
\end{align}
Yet since the $Z_{k, 0}$ are pairwise independent and $Z_{k, 0} \sim \text{Poisson}(\theta/k),$ we can interchange the product and expectation operator and then compute the factorial moments to get
\begin{align}
(5.20) &= \left(\sum_{j = 0}^{\infty} \frac{(1+2j)\E(Z_{m, 0})_{2j}}{(2m)^jj!}\right)\!\prod_{\substack{1 \leq k \leq n \\ k \neq m}}\sum\limits_{j=0}^{\infty} \frac{(\theta/k)^{2j}}{(2k)^{j}j!}
\\ &= \left(\sum_{j = 0}^{\infty} \frac{(1+2j)(\theta/m)^{2j}}{(2m)^jj!}\right)\!\prod_{\substack{1 \leq k \leq n \\ k \neq m}}\sum\limits_{j=0}^{\infty} \frac{(\theta/k)^{2j}}{(2k)^{j}j!}
\\ &= \left(1 + \frac{\theta^2}{m^3}\right)\!e^{\frac{\theta^2}{2m^3}}\prod_{\substack{1 \leq k \leq n \\ k \neq m}} e^{\frac{\theta^2}{2k^3}} \leq (1 + \theta^2)e^{\frac{\theta^2}{2}\zeta(3)}.
\end{align}
Therefore $\PP(|S_{2, n} - S_{1, n}| > \epsilon) = O(n^{-\sqrt{\theta/3}}),$ and so
\begin{equation}
\sup_{x \in \R} |\PP(S_{2, n} \leq x) - \Phi(x)| \leq \frac{c + C_{\Phi}}{\sqrt{\log n}} + O\!\left(n^{-\sqrt{\frac{\theta}{3}}}\right) = O\!\left(\frac{1}{\sqrt{\log n}}\right).
\end{equation}
\end{proof}

According to a result of \cite{Zacharovas},
\[\sup_{x \in \R} \left|\PP\!\left(\frac{\log\B_n(C^{(n)}) - \frac{1}{2}\log^2\!n}{\sqrt{\frac{1}{3}\log^3\!n}} \leq x\right) - \Phi(x) - \frac{3^{3/2}}{24\sqrt{2\pi}}\frac{(1 - x^2)e^{-x^2/2}}{\sqrt{\log n}}\right| = O\!\left(\!\left(\frac{\log\log n}{\log n}\right)^{\frac{2}{3}}\right).\]
Putting into (\ref{closeone}),
\[U = \frac{\log\B_n(C^{(n)}) - \frac{1}{2}\log^2\!n}{\sqrt{\frac{1}{3}\log^3\!n}}, \hspace{2.5em} X = -\frac{\log\B_n(C^{(n)}) - \log\mathsf{invol}_n(C^{(n)})}{\sqrt{\frac{1}{3}\log^3 n}},\]
$\epsilon = ((\log\log n)/\log n)^{2/3},$ $\eta = \log^{-3/4}\!n,$ and observing that
\begin{equation}
\PP(|S_{2, n} - S_{1, n}| \geq \epsilon) \leq (1 + \theta^2)e^{\frac{\theta^2}{2}\zeta(3) - \epsilon\sqrt{\frac{\theta}{3}\log^3\!n}} = O\!\left(e^{-(\log\log n)^{2/3}\log^{5/6}\!n}\right)
\end{equation}
yields the following.

\begin{theorem}
If $C^{(n)}$ is distributed according to $\textup{ESF}(1),$ then
\[\sup_{x \in \R} \left|\PP\!\left(\frac{\log\mathsf{invol}_n(C^{(n)}) - \frac{1}{2}\log^2\!n}{\sqrt{\frac{1}{3}\log^3\!n}} \leq x\right) - \Phi(x) - \frac{3^{3/2}}{24\sqrt{2\pi}}\frac{(1 - x^2)e^{-x^2/2}}{\sqrt{\log n}}\right| = O\!\left(\!\left(\frac{\log\log n}{\log n}\right)^{\frac{2}{3}}\right).\]
\end{theorem}

Zacharovas's result on the optimality of the convergence rate for the Erd\H{o}s-Tur\'{a}n limit law is constructive and can be mimicked for $\theta$-weighted permutations in general. For this reason, the rate of convergence stated in Theorem \ref{errorrate} is the best possible. The proof is excluded here, but all of the relevant details can be found in \cite{Zacharovas}.

Although we decline from doing so here, we remark that it is straightforward to compute explicit bounds for the constants implied by the error terms in Theorems \ref{functional} and \ref{errorrate}. The estimates needed to measure the distance between the distributions of $\log\B_n(C^{(n)})$ and $\log\B_n(Z)$ already appear in \cite{BT}. From there, a Berry-Ess\'{e}en inequality will disclose how quickly $\log\B_n(Z)$ tends to its Gaussian limit law. The constants implied by (\ref{distance1}) and (\ref{distance2}) require more delicate asymptotic evaluation, but can be safely ignored since the dominant error term for Theorem \ref{functional} resides in Theorem \ref{brownian}.

\section*{Concluding Remarks}

The nonuniformity of $\mathsf{invol}$ makes it evident that the composition of two uniformly random involutions is not a uniformly random permutation. Furthermore, if we let $t_n$ denote the total number of involutions in $\mathfrak{S}_n,$ then
\begin{equation}
t_n^2 \sim n!\cdot\frac{e^{2\sqrt{n}}}{\sqrt{8\pi en}} \sim \left|\mathfrak{S}_n\right|\cdot\E_{1, n}\mathsf{invol}.
\end{equation}
(Refer to pg. 583, Example VIII.9, \cite{FS}.) Therefore, using pairs of random involutions to generate random permutations introduces a significant sampling bias.

Like $\text{ESF}(\theta)$ for $\theta > 1,$ a composition of two random involutions is likelier to have a lot of cycles. However, the distribution of $c_k$ looks quite different in this model. The factors of the exponential generating function in (\ref{bgf}) suggests that $c_k$ converges in law to the distribution of $X + 2Y_k,$ where $X$ and $Y_k$ are independent Poisson random variables of mean 1 and $1/(2k),$ respectively.  \cite{LugoCompose} proved as much in an unpublished manuscript. The author intends to expand on these and other observations in a followup paper.

\acknowledgements
This paper refines various elements of the author's doctoral thesis. The author wishes to thank his advisor, Eric Schmutz, for his extensive guidance. The author also expresses his gratitude to Dirk Zeindler for first suggesting to him that the results of \cite{BS} can be generalized to the Ewens Sampling Formula, to Hsien-Kuei Hwang for helpful discussions on how to tackle the variance of $\mathsf{invol},$ and to Eugenijus Manstavi\v{c}ius for reaching out and sharing some insightful background on this topic.

\section{Appendix}
\label{sec:appendix}
This appendix compiles some technical prerequisites which are used sporadically throughout the manuscript.

\subsection{Wright's Expansions}

The following consequence of the saddle-point method applied to functions with exponential singularities is often helpful in unraveling the asymptotic development of various random permutation statistics.

\begin{theorem}
\label{Wrightgeneral}
\textup{(\cite{Wright2})}
Suppose
\[f(z) = (1 - z)^{-\beta}\phi(z)e^{P(z)},\]
where $\phi$ is regular in the unit disc, $\phi(1) \neq 0,$
\[P(z) = \sum_{m = 1}^{M} \frac{\alpha_m}{(1 - z)^{\rho_m}},\]
$\rho_1, \ldots, \rho_M$ are real, $\beta, \alpha_1, \ldots, \alpha_M$ are complex, $\rho_1 > 0,$ and $\alpha_1 \neq 0.$ If in addition
\[P(e^{-u}) = Au^{-\rho}\left(1 + \sum_{\ell = 1}^{L} A_{\ell}u^{\sigma_{\ell}}\right) + O(u^K), \hspace{2.5em} u \to 0^+,\]
where $ A \in \C - \R_{\leq 0},$ $0 < \sigma_1 < \sigma_2 < \cdots < \sigma_L \leq \rho,$ and $L,$ $A,$ $\rho,$ the $\sigma_{\ell},$ and the $A_{\ell}$ are all calculable in terms of the $\rho_m$ and $\alpha_m,$ then
\[[z^n]f(z) = \frac{N^{\beta - \frac{1}{2}}e^{\Omega(N)}\phi(1)}{\sqrt{2\pi n(\rho + 1)}}\left(1 + O(n^{-K})\right), \hspace{2.5em} n \to \infty,\]
where
\begin{align*}
N &= \left(\frac{n}{A\rho}\right)^{\frac{1}{\rho + 1}},
\\ \Omega(N) &= A(1 + \rho)N^{\rho} + A\sum_{\ell = 1}^{L} A_{\ell}N^{\rho - \sigma_{\ell}} + \sum_{s = 2}^{\infty} \frac{(-1)^{s - 1}(A\rho)^{1 - \frac{\rho}{\rho + 1}}}{s!(\rho + 1)}\left.\left[\frac{d^{s - 2}}{dy^{s - 2}}(y^{\frac{\rho}{\rho + 1}-2}(\vartheta(y))^s)\right]\right|_{y = n}, \text{and}
\\ \vartheta(y) &= \frac{1}{\rho}\sum_{\ell = 1}^{L} (\rho - \sigma_{\ell})A_{\ell}\left(\frac{y}{A\rho}\right)^{-\frac{\sigma_{\ell}}{\rho + 1}}.
\end{align*}
\end{theorem}

With one exception in section 3, we solely need this simpler instance of Theorem \ref{Wrightgeneral}.

\begin{corollary}
\label{Wrightasymptotic}
\textup{(\cite{Wright1})}
The leading-term asymptotic for
\[s_n = [z^n](1 - z)^{\beta}\phi(z)\exp\left(\frac{\alpha}{1 - z}\right),\]
where $\beta$ is a complex number, $\phi$ is regular in the unit disc, $\phi(1) \neq 0,$ and $\alpha$ is a nonzero real number, is given by
\[s_n = \frac{1}{n^{\frac{\beta}{2} + \frac{3}{4}}}\left[\frac{\exp\!\left(2\sqrt{\alpha n}\right)}{2\sqrt{\pi}}\phi(1)e^{\frac{\alpha}{2}}\alpha^{\frac{\beta}{2} + \frac{1}{4}}\right]\!\left(1 + O(n^{-\frac{1}{2}})\right).\]
\end{corollary}

\subsection{The Mellin Transform}
\label{sec:mellin}
To conclude this section, we review some pertinent facts about the Mellin transform
as a means of extracting asymptotic estimates. Given a locally integrable function $f$
defined over $\R_{> 0},$ the Mellin transform of $f$
is the complex-variable function
\[f^{\star}(s) := \int_{0}^{\infty} f(x)x^{s - 1}\,dx.\]
The transform is also occasionally denoted by $\mathcal{M}[f]$ or $\mathcal{M}[f(x); s].$
If $f(x) = O(x^{-a})$ as $x \to 0^+$ and $f(x) = O(x^{-b})$ as $x \to +\infty,$
then $f^*$ is an analytic function on the ``fundamental strip'' $a < \Re(z) < b$
and is often continuable to a meromorphic function on the whole of $\C.$

Notably, the Mellin transform is a linear operator that satisfies the rescaling rule
\[\mathcal{M}[f(\mu x); s] = \mu^{-s}f^{\star}(s),\]
for any $\mu > 0.$ This together with the linearity of $\mathcal{M}$ yields the ``harmonic sum rule''
\[\mathcal{M}\!\left[\sum_k \lambda_k f(\mu_k x); s\right] = \left(\sum_k \lambda_k \mu_k^{-s}\right)\!f^{\star}(s),\]
provided that the integral and summation are interchangeable. Sums of the form $\sum_k \lambda_kf(\mu_k x)$
are known as harmonic sums with base function $f,$ amplitudes $\lambda_k,$ and frequencies $\mu_k.$
In other words, the Mellin transform of a harmonic sum factors into the product of the Mellin transform of the
base function and a generalized Dirichlet series associated with the corresponding sequence of amplitudes and frequencies.
For instance, we can obtain the useful Mellin pairs
\begin{equation}
\label{Mellinexample1}
\mathcal{M}\left[\frac{e^{-x}}{1 - e^{-x}};s\right] = \mathcal{M}\left[\sum_{k = 1}^{\infty} e^{-kx}; s\right] = \left(\sum_{k = 1}^{\infty} k^{-s}\right)\mathcal{M}[e^{-x}; s] = \zeta(s)\Gamma(s)
\end{equation}
and
\begin{equation}
\label{Mellinexample2}
\mathcal{M}\left[\frac{e^{-x}}{(1 - e^{-x})^2};s\right] = \mathcal{M}\left[\sum_{k = 1}^{\infty} ke^{-kx};s\right] = \left(\sum_{k = 1}^{\infty} k^{-s + 1}\right)\mathcal{M}[e^{-x}; s] = \zeta(s - 1)\Gamma(s).
\end{equation}

Akin to many other valuable integral transforms, the Mellin transform is essentially involutive.
If $f$ is continuous in an interval containing $x,$ then $f(x)$ can be recovered from
$f^{\star}(s)$ via the inversion formula
\[f(x) = \frac{1}{2\pi i}\int_{c - i\infty}^{c + i\infty} f^{\star}(s)x^{-s}\,ds,\]
where the abscissa $c$ must be chosen from the fundamental strip of $f.$ If the continuation of $f^{\star}(s)$
decays quickly enough as $\Im(c) \to +\infty,$ then the inverse Mellin integral and Cauchy's residue theorem
can be combined to derive an asymptotic series for $f(x)$ as $x \to 0^+.$ We have the following
transfer rule in particular: a pole of order $k + 1$ at $s = s_0$ with
\[f^{\star}(s) \underset{s \to s_0}{\sim} a_{s_0}\frac{(-1)^kk!}{(s - s_0)^{k + 1}}\]
translates to the term $a_{s_0}x^{-s_0}(\log x)^k$ in the singular expansion of $f(x)$ near
$x = 0.$ (The specific requirements for this transfer rule are all fulfilled by the special functions appearing in this manuscript,
as the gamma function decays rapidly along vertical lines in the complex plane whereas polylogarithms undergo only moderate growth.) Detailed explanations along with further uses of the
Mellin transform towards solving asymptotic enumeration problems can be found in \cite{FGD} and \cite{FS}.

\nocite{*}
\bibliographystyle{abbrvnat}
\bibliography{Involution_factorizations_of_Ewens_Random_Permutations}
\label{sec:biblio}

\end{document}